\documentclass[9pt]{amsart}
\usepackage{amssymb}
\usepackage{graphicx}
\usepackage{xcolor} 
\usepackage{tensor}
\usepackage[
margin=1.45cm,
includefoot,
footskip=25pt,
]{geometry}

\usepackage{amsmath}
\usepackage{amsthm}
\usepackage{verbatim}
\usepackage{subcaption}
\usepackage{hyperref}

\usepackage{enumitem}
\setlist[enumerate]{leftmargin=1.2em}
\setlist[itemize]{leftmargin=1.2em}

\usepackage{seqsplit,mathtools}

\definecolor{green}{rgb}{0,0.8,0} 

\newtheorem{theorem}{Theorem}[section]

\newtheorem{lemma}[theorem]{Lemma}
\newtheorem{proposition}[theorem]{Proposition}

\theoremstyle{definition}

\theoremstyle{remark}
\newtheorem{remark}[theorem]{Remark}

\numberwithin{equation}{section}
\newcommand{\nrm}[1]{\Vert#1\Vert}

\newcommand{\nnrm}[1]{{\vert\kern-0.25ex\vert\kern-0.25ex\vert #1 
		\vert\kern-0.25ex\vert\kern-0.25ex\vert}}

\newcommand{\dist}{\mathrm{dist}\,}

\newcommand{\lap}{\Delta}

\newcommand{\ud}{\mathrm{d}}
\newcommand{\rd}{\partial}
\newcommand{\nb}{\nabla}

\newcommand{\alp}{\alpha}

\newcommand{\gmm}{\gamma}

\newcommand{\dlt}{\delta}

\newcommand{\eps}{\epsilon}

\newcommand{\tht}{\theta}

\newcommand{\omg}{\omega}
\newcommand{\Omg}{\Omega}

\newcommand{\bbN}{\mathbb N}

\newcommand{\bbR}{\mathbb R}

\newcommand{\bbT}{\mathbb T}

\newcommand{\bbZ}{\mathbb Z}

\newcommand{\sd}{\triangle}

\begin{document}
\bibliographystyle{plain}
 \title{Growth of perimeter for vortex patches in a bulk}
 
 \author{Kyudong Choi}
 \address{UNIST, Ulsan, Korea 44919}
 \email{kchoi@unist.ac.kr}
 
 \author{In-Jee Jeong}%
 \address{KIAS, Seoul, Korea 02455}%
 \email{ijeong@kias.re.kr}%
   
\date\today
 
  \maketitle

\renewcommand{\thefootnote}{\fnsymbol{footnote}}
\footnotetext{\emph{Key words:} 2D Euler, vortex patch, stability, perimeter,  large time behavior, particle trajectory.
\quad\emph{2010 AMS Mathematics Subject Classification:} 76B47, 35Q35 }
\renewcommand{\thefootnote}{\arabic{footnote}}



\begin{abstract}
We consider the two-dimensional incompressible  Euler equations. 
We construct vortex patches with smooth boundary on $\bbT^2$ and $\bbR^2$ whose perimeter grows with time. 
More precisely, for any constant $M>0$, we construct a vortex patch in $\bbT^2$ whose  smooth boundary has length of order $1$ at the initial time such that   the perimeter grows  up to the given constant $M$ within finite time. The construction is done by
cutting a thin stick out of an almost square patch.  A similar result holds   for an almost round patch with a thin handle in  $\bbR^2$.

\end{abstract}

\section{Introduction}
\noindent We consider the  incompressible Euler equation in vorticity form: \begin{equation}\label{eq:Euler}
\begin{split}
\rd_t\omg + u\cdot\nb\omg = 0,\quad u=\nb^\perp\lap^{-1}\omg, 
\end{split}
\end{equation} where $\nb^\perp = (-\rd_{x_2},\rd_{x_1})$. We take the physical domain to be either $\mathbb{T}^2 = (\bbR/2\pi\bbZ)^2$ or $\bbR^2$. In the former case, we need to assume that $\int_{\mathbb{T}^2}\omega_0\,dx=0$ holds, which propagates in time. At each moment of time, the velocity $u(t)$ can be obtained from the vorticity $\omg(t, \cdot )$ via the Biot-Savart law: \begin{equation*}
\begin{split}
u(t,x) = \int_{\bbR^2} K(x-y)\omega(t,y) \,  \ud y , \quad K(x) =\frac{1}{2\pi} \frac{x^\perp}{|x|^2}. 
\end{split}
\end{equation*} The above formula holds in the case of $\bbT^2$ as well, by periodically extending $\omega(t)$ to all of $\bbR^2$.

For $\omg_0\in L^\infty(\bbT^2)$ or $L^\infty\cap L^1(\bbR^2)$, the well-known result of Yudovich (\cite{Y1}) shows that there exists a unique and global-in-time solution to \eqref{eq:Euler}. In particular, when the initial data takes the form  $\omg_0 = \sum_{i=1}^{n} a_i {\mathbf{1}}_{\Omega_0^i}$  for open sets   $\Omg_0^i$  with disjoint closures and constants $a_i\in\bbR$, it can be shown that the unique solution is given by  $\omg(t)= \sum_{i=1}^{n} a_i {\mathbf{1}}_{\Omega_t^i}$  where $\Omg_t^i$ again has disjoint closures with each other for any $t>0$. Solutions of this form are commonly referred to as vortex patches. Introducing the flow $\Phi(t,x)$ satisfying\begin{equation*}
\begin{split}
\frac{d}{dt}\Phi(t,x)=u(t,\Phi(t,x)),\quad \Phi(0,x)=x,
\end{split}
\end{equation*} we have that $\Omg_t^i = \Phi(t,\Omg_0^i)$ for each $1\le i \le n$. Given $x$, the curve $t\mapsto \Phi(t,x)$ describes the trajectory of the ``fluid particle'' initially located at $x$. 

It is known that if initially the patch boundary is smooth, then the smoothness propagates for all time (\cite{C3,BeCo,S2}). The estimates give that the perimeter of the patch could grow at most double exponentially in time (\cite{BeCo,NK}). When one considers patch solutions in domains with a boundary, linear in time growth of the perimeter can be easily achieved for certain initial data. For instance, inside an annulus one can consider a patch which connects two boundary components. It can be shown that for all times, the inner and outer boundaries rotate with different angular velocities (\cite{Nad}), which in particular stretches the patch boundary at least linearly in time. A similar effect can be obtained if one sets up two patches of different sign to touch each other (\cite{SVP2}). When boundaries are absent and patches are separated from each other, it seems like a difficult question to show growth of perimeter in time. In this paper, we prove that for any $M>0$, there exists a patch having smooth boundary and uniformly bounded perimeter in $\bbT^2$ and $\bbR^2$ so that the perimeter becomes larger than $M$ at a later time. Our constructions in $\bbT^2$ and $\bbR^2$ are based on stability results for the Bahouri-Chemin steady state (\cite{BC}) and the circular patch (\cite{wp,SiVe}), respectively. The latter construction partially confirms the well known phenomena called \textit{filamentation}--nearly circular patches forming ``long arms'' (\cite{CS,Drit,PuMo}). We note that the Bahouri-Chemin solution has been very useful in proving small scale creation in the 2D Euler equation (\cite{Den2,Den3,KS,EJ,Xu,Z,KiLi}). 
To prove growth of the perimeter, the main task is to measure the length of a curve in a bulk given by Euler evolution at a certain time; this is much harder than just measuring the length of the trajectory of a particle up to a given time. 
\\

\subsection{Main result}  

We now state our main results, which are sufficient conditions for an initial patch to guarantee at least linear growth of the perimeter up to some finite but arbitrarily large time. 
We denote the origin by $x_o = (0,0)$ and the center of the square $[0,\pi]^2$ by  $x_c = (\frac{\pi}{2},\frac{\pi}{2})$.
Moreover, given $x$ and $r>0$, we use $B(x,r)$ to denote the open ball of radius $r$ centered at $x$. 
In the $\bbT^2$ case, we shall always assume that $\omg$ is odd with respect to $x_1$ and $x_2$,\footnote{The 2D Euler equation preserves odd symmetry; if $\omg_0$ satisfies $\omg_0(x_1,x_2)=-\omg_0(-x_1,x_2)$ and $\omg_0(x_1,x_2)=-\omg_0(x_1,-x_2)$ for all $x_1,x_2\in\bbT$, the solution $\omg(t)$ again satisfies  $\omg(t,x_1,x_2)=-\omg(t,-x_1,x_2)$ and $\omg(t,x_1,x_2)=-\omg(t,x_1,-x_2)$ for all $t>0$.} and we shall specify the data and the solution only on $[0,\pi]^2$. 

Before presenting the main result, let us define the notion of the winding number around the point $x_c$ for any continuous curve $\gmm:[0,1]\rightarrow (0,\pi)^2\backslash\{x_c\}\subset (0,\infty)\times (\bbR/2\pi\bbZ)$. The last inclusion is provided by the polar coordinates $(r,\tht)$ centered at $x_c$: $r$ and $\tht$ are defined by $r = |x-x_c|$ and $\tan(\tht(x))= \frac{x_2-x_{c,2}}{x_1-x_{c,1}}$, respectively. Any such curve can be lifted to a continuous curve $\tilde{\gmm}=(\tilde{\gmm_1},\tilde{\gmm_2}):[0,1]\rightarrow (0,\infty)\times \bbR$  where $\Pi \circ \tilde{\gmm} = \gmm$ with $\Pi:\bbR\rightarrow\bbR/2\pi\bbZ$ being a natural projection. Then we define the winding number $N_{x_c}[\gmm]$ by
\begin{equation}\label{defn_wind}
N_{x_c}[\gmm]=\frac 1 {2\pi} \left(\tilde{\gmm_2}(1) -\tilde{\gmm_2}(0)\right),
\end{equation} i.e.
   the difference of the ``angular'' component of $\tilde{\gmm}(1)$ and $\tilde{\gmm}(0)$ multiplied by $(2\pi)^{-1}$
(cf. see (\cite{CJ_winding, Choi_distance}) for  particle trajectories).     Note that if the image of $\gmm$ is contained in $ (r_0,\infty)\times (\bbR/2\pi\bbZ)$ (in the polar coordinates) for some $r_0>0$, then 
\begin{equation}\label{result_wind}
 \mathrm{length}(\gmm)\ge 2\pi r_0 \cdot|N_{x_c}[\gmm]|.
\end{equation}

Our main result in the $\bbT^2$ case roughly states that if the initial patch fills up most of $(0,\pi)^2$ but does not contain some region near $x_c$, then growth of the perimeter must occur.
\begin{theorem}[The torus case]\label{thm_finite}
	For any small $\epsilon>0$ and for any open set $\Omg_0$ compactly contained in $(0,\pi)^2$ 
 with	its smooth boundary $\partial\Omega$ 
	satisfying \begin{itemize}
		\item $\mathrm{area}([0,\pi]^2\backslash \Omg_0)\le\eps^2$, and 
		\item there exists a continuous curve $\gamma$ lying on  $\partial\Omg_0$, intersecting $B(x_o,\widetilde\eps)$ and $B(x_c,\dlt)$ where $\widetilde{\eps} = (c_2\dlt)^{\exp((c_3\eps)^{-1})}$, and satisfying
		$$|N_{x_c}[\gamma]|\leq 1,$$
	\end{itemize}   the solution $\mathbf{1}_{\Omg_t}$ to \eqref{eq:Euler} with initial data $\mathbf{1}_{\Omg_0}$ satisfies \begin{equation*}
\begin{split}
\mathrm{length}( \partial \Omg_t ) \ge c_0 t \quad \mbox{ for all } \quad t \in [0,\frac{c_1}{\eps}].
\end{split}
\end{equation*} Here, $\dlt, c_0, c_1, c_2$ and $c_3$ are small positive constants which do not depend on $\eps$. 
\end{theorem}

In the $\bbR^2$ case,  for any continuous curve $\gamma:[0,1]\to \mathbb{R}^2\setminus\{x_0\}$, 
we can similarly define the winding number $N_{x_0}[\gmm]$ around the origin $x_0$  as in \eqref{defn_wind}. Then, 
we have growth of the perimeter for patches which are nearly circular in the $L^1$ sense but are with a ``handle".
\begin{theorem}[The whole space case]\label{thm_finite_whole}
	For any  small $\epsilon>0$
and	
	for any open set $\Omg_0\subset \bbR^2$ 
 with	its smooth boundary $\partial\Omega$ 	
	satisfying 
	\begin{itemize}
		\item $\Omg_0\subset B(x_o,3)$ and $\mathrm{area}(\Omg_0 \sd B(x_o,1))\le
c_1\epsilon^4$,		
		and 
		\item there exists 
a continuous curve $\gamma$ lying on 		
		 $\partial\Omg_0$, intersecting $\{|x|=i\}$ for $i=1,2,$
		 and satisfying
		 	$$|N_{x_o}[\gamma]|\leq 1,$$
	\end{itemize}
	the solution $\omg_t=\mathbf{1}_{\Omg_t}$  to \eqref{eq:Euler} with initial data $ {\omg_0}=\mathbf{1}_{\Omg_0}$ satisfies
	\begin{equation*}
	\begin{split}
	\mathrm{length}( \partial \Omg_t ) \ge c_0 t \quad \mbox{ for all } \quad t \in [0,\frac{ 1}{\eps}].
	\end{split}
	\end{equation*} Here, $c_0$ and $c_1$  are  small positive constants which does not depend on $\eps$. 
\end{theorem}  

\begin{remark}
	It is not difficult to see that for any given $\eps>0$, there exists  $\Omg_0$ satisfying the assumptions of Theorem \ref{thm_finite} and \ref{thm_finite_whole} with 
	 uniformly bounded perimeter, say $\mathrm{length}( \partial \Omg_0 ) \le 20$. For instance, see Figure  \ref{figs}.  
	

\begin{figure}[h]

\begin{subfigure}{0.45\textwidth}
\includegraphics[scale=0.52]{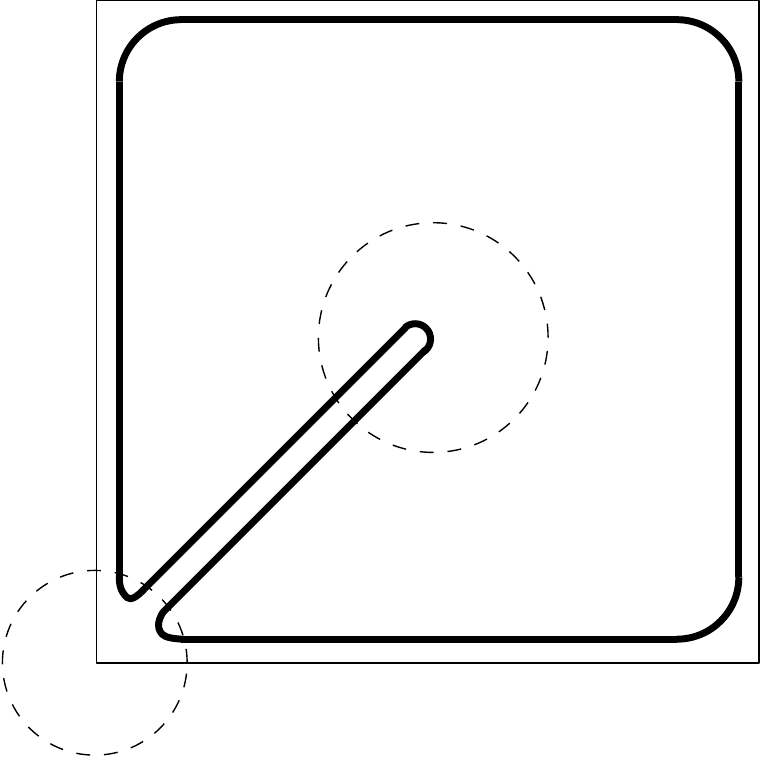} 
\caption{
$\mathbb{T}^2$-case, 
 depicted on $[0,\pi]^2$. The dashed circles represent $B(x_o,\widetilde{\eps})$ and $B(x_c,\dlt)$.}
\label{fig:torus}
\end{subfigure}
\begin{subfigure}{0.45\textwidth}
\includegraphics[scale=0.4]{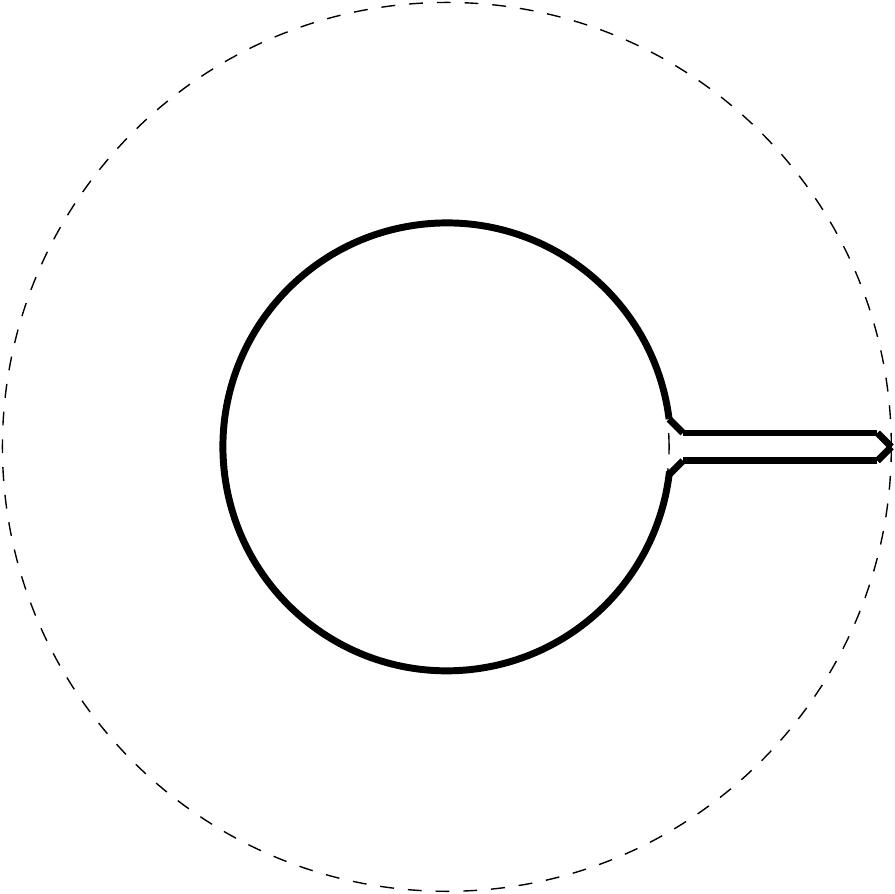}
\caption{
$\mathbb{R}^2$-case.
The dashed circle represents $B(x_o,2)$.}
\label{fig:whole}
\end{subfigure}
\caption{Examples of patches satisfying the assumptions of Theorems \ref{thm_finite}, \ref{thm_finite_whole}}
\label{figs}
\end{figure}
\end{remark}

\begin{remark}
	While infinite-time linear perimeter growth should be achievable if $\overline{\partial\Omg_0}$ is allowed to intersect with $\partial([0,\pi]^2)$ in the setup of Theorem \ref{thm_finite}, it seems like a challenging problem to lift such an additional condition. 
\end{remark}
We prove Theorems \ref{thm_finite} and \ref{thm_finite_whole} respectively in Sections \ref{sec:T2} and \ref{sec:R2}. In both cases, the strategy is to use stability properties of the background solution (which is the Bahouri-Chemin solution and the circular vortex patch, respectively) and arrange two parts of the patch to rotate with different angular speeds, which results in an increase of the winding number of the patch boundary. As a consequence of \eqref{result_wind},  the length of the boundary grows in time. \\

In the proof, when $\gmm_t=\Phi(t,\gamma_0)$, the increment (in time) of the winding number of the curve $\gmm_t$ around a fixed point $x$ is accessed by the following way: \\
   \begin{equation}\label{iden_wind}
   \begin{split}
{2\pi}(N_{x}[\gmm_t]-N_{x}[\gmm_0])&=
  \left({\tilde\gmm_{t,2}}(1) -{\tilde\gmm_{t,2}}(0)\right)-
  \left({\tilde\gmm_{0,2}}(1) -{\tilde\gmm_{0,2}}(0)\right)=
  \left({\tilde\gmm_{t,2}}(1) -{\tilde\gmm_{0,2}}(1)\right)-
  \left({\tilde\gmm_{t,2}}(0) -{\tilde\gmm_{0,2}}(0)\right)\\
&= \left( \int_0^t u(s,\gmm_s(1))\cdot \frac{(\gmm_s(1)-x)^\perp}{|\gmm_s(1)-x|^2}ds-\int_0^t u(s,\gmm_s(0))\cdot \frac{(\gmm_s(0)-x)^\perp}{|\gmm_s(0)-x|^2}ds\right).
\end{split}
\end{equation}  
	Each integral in \eqref{iden_wind} represents the winding number (in time) around $x$ of each particle trajectory $\gmm_t(1), \gmm_t(0)$ up to time $t$  (e.g. see \cite{CJ_winding}).

\section{The torus case}\label{sec:T2}

\subsection{Bahouri-Chemin stationary solution}\label{subsec_bc}

The Bahouri-Chemin steady state is given by (\cite{BC}) \begin{equation*}
\begin{split}
\omg^{BC}(x_1,x_2) = \mathrm{sgn}(x_1)\mathrm{sgn}(x_2), \quad -\pi<x_1,x_2\le \pi. 
\end{split}
\end{equation*} Note that $\omg^{BC}$ is odd in both variables. This defines a stationary solution to \eqref{eq:Euler} in $\bbT^2$, simply because the particle trajectories cannot cross the separatrices, to preserve odd symmetry. Introducing the stream function $\psi^{BC}$ defined by \begin{equation*}
\begin{split}
\lap\psi^{BC} = \omg^{BC}\quad  &\mbox{in} \quad (0,\pi)^2 \\
\psi^{BC} = 0  \quad &\mbox{on} \quad \partial([0,\pi]^2),
\end{split}
\end{equation*} the velocity is given by $u^{BC} = \nb^\perp \psi^{BC}$ and we have that the particle trajectories lie on level sets of $\psi^{BC}$, which are closed curves foliating the torus. As a warm-up, we 
construct an example of a curve whose length grows linearly with time under the Bahouri-Chemin flow.
 The properties of $\psi^{BC}$ established here will be useful in the proof of our main result. 
\begin{proposition}\label{prop:BC}
	Let $\gmm:[0,1]\rightarrow [0,\pi]^2$ be the   line segment 
	$$\gamma(t)=(1-t)x_0+tx_c,\quad t\in[0,1].$$
	 Let $\Phi^{BC}:[0,\infty) \times [0,\pi]^2 \rightarrow [0,\pi]^2$ be the flow map generated by $\omg^{BC}$ restricted to $[0,\pi]^2$. Then, \begin{equation*}
	\begin{split}
	\mathrm{length}(\Phi^{BC}(t,\gmm)) \ge c_0 t
	\end{split}
	\end{equation*} for all $t>0$, with some absolute constant $c_0>0$. 
\end{proposition}
\begin{proof}
	Since $\Phi^{BC}(t,\cdot):[0,\pi]^2 \rightarrow [0,\pi]^2$ defines a diffeomorphism away from the points $(0,0)$, $(0,\pi)$, $(\pi,0)$, $(\pi,\pi)$, $\gmm_t:=\Phi^{BC}(t,\gmm)$ is a rectifiable curve satisfying $\gmm_t(0)=x_0$ and $\gmm_t(1)=x_c$ for any $t>0$. 
	
	Note that $\psi^{BC}$ is $C^\infty$-smooth and actually analytic away from $\partial ([0,\pi]^2)$. We consider the Taylor expansion of $\psi^{BC}$ around $x_c$: we have \begin{equation*}
	\begin{split}
	\psi^{BC}(x) = \psi^{BC}(x_c) + \frac{1}{4}|x-x_c|^2 + O(|x-x_c|^4). 
	\end{split}
	\end{equation*} In the above, the linear and cubic terms in $x-x_c$ do not appear due to the 4-fold symmetry of $\psi^{BC}$ around $x_c$; that is, $\psi^{BC}$ is invariant under the counterclockwise $\frac{\pi}{2}$-rotation around $x_c$. Therefore, we have the following expansion for $u^{BC}$: \begin{equation*}
	\begin{split}
	u^{BC}(x) = \frac{1}{2}\begin{pmatrix}
	-(x_2-x_{c,2})\\
	x_1-x_{c,1}
	\end{pmatrix} + O(|x-x_c|^3). 
	\end{split}
	\end{equation*} We shall fix an absolute constant $0<\dlt\le\frac{1}{4}$ such that for $|x-x_c|\le 2\dlt$, \begin{equation}\label{eq:u-bounds}
	\begin{split}
	\left| u^{BC}(x) - \frac{1}{2}\begin{pmatrix}
	-(x_2-x_{c,2})\\
	x_1-x_{c,1}
	\end{pmatrix} \right| \le \frac{1}{4}|x-x_c|^2, 
	\end{split}
	\end{equation} and \begin{equation}\label{eq:psi-bounds}
	\begin{split}
	\left| \psi^{BC}(x) -\psi^{BC}(x_c) - \frac{1}{4}|x-x_c|^2 \right| \le \frac{1}{16}|x-x_c|^3. 
	\end{split}
	\end{equation} 
	
	Assume that $x$ satisfies $\frac{1}{2}\dlt \le |x-x_c|$. Then, it follows from \eqref{eq:psi-bounds} that \begin{equation*}
	\begin{split}
	\frac{1}{4}\dlt \le |\Phi^{BC}(t,x)-x_c|
	\end{split}
	\end{equation*} for all $t>0$. Now we restrict $\gmm$ to $[0,a]\subset [0,1]$ such that $\widetilde{\gmm}:=\gmm|_{[0,a)}$ does not intersect the (open) ball $B(x_c,\frac{\dlt}{2})$ and $|\gmm(a)-x_c| = \frac{\dlt}{2}$. Then, $\Phi^{BC}(t,\widetilde{\gmm})$ does not intersect $B(x_c,\frac{\dlt}{4})$ for all $t>0$. Set $x^*=\gmm(a)$ and let $T>0$ be the period of $x^*$; that is, \begin{equation*}
	\begin{split}
	T = \min \{ t>0: \Phi^{BC}(t,x^*)=x^*  \}. 
	\end{split}
	\end{equation*} (The fact that $T<+\infty$ follows from \eqref{eq:u-bounds} and $|x_c-x^*| = \frac{1}{2}\dlt$.) Then, the curve $\Phi^{BC}(t,\widetilde{\gmm})$ winds around the ball $B(x_c,\frac{1}{4}\dlt)$ at least $\lfloor \frac{t}{T} \rfloor$ times, where $\lfloor \alp\rfloor$ denotes the largest integer not exceeding $\alp$. Hence \begin{equation*}
	\begin{split}
	\mathrm{length}(\Phi^{BC}(t,\gmm)) \ge \mathrm{length}(\Phi^{BC}(t,\widetilde{\gmm})) \ge \frac{\pi\dlt}{2}\left\lfloor \frac{t}{T} \right\rfloor. 
	\end{split}
	\end{equation*} Since
$\mathrm{length}(\Phi^{BC}(t,\gmm)) \ge	\sqrt 2 \cdot \pi/2$ for all $t>0$,
	this gives the desired lower bound. 
\end{proof}

\subsection{Proof of Theorem \ref{thm_finite}}

We take and fix $\dlt>0$ from the proof of Proposition \ref{prop:BC} and take $\Omg_0\subset (0,\pi)^2$ satisfying the assumptions of Theorem \ref{thm_finite}. Let us denote the solution (restricted to $(0,\pi)^2$) by $\omg(t,x)=\mathbf{1}_{\Omg_t}$. The corresponding velocity and flow is denoted respectively by $u(t,x)$ and $\Phi(t,x)$. Let us prove a few lemmas which work in this setup.
\begin{lemma}\label{lem:stab}
	We have that $$|u(t,x)-u^{BC}(x)|\le C \eps$$  for any $t>0$ and $x\in [0,\pi]^2$ where $C>0$ is a universal constant. 
\end{lemma}
\begin{proof}
	Since $u(t,\cdot)$ is divergence-free, $\mathrm{area}([0,\pi]^2\backslash\Omg_t) = \mathrm{area}([0,\pi]^2\backslash\Omg_0)\le\eps^2$ for all $t\ge0$. From the elementary estimate $\nrm{\nb \lap^{-1}f}_{L^\infty(\bbT^2)} \le C\nrm{f}_{L^\infty(\bbT^2)}^{\frac{1}{2}}\nrm{f}_{L^1(\bbT^2)}^{\frac{1}{2}}$ (see e.g. \cite{MP}), \begin{equation*}
	\begin{split}
	\nrm{u(t)-u^{BC}}_{L^\infty} \le C\nrm{\omg(t)-\omg^{BC}}_{L^\infty}^{\frac{1}{2}}\nrm{\omg(t)-\omg^{BC}}_{L^1}^{\frac{1}{2}}\le C\eps. 
	\end{split}
	\end{equation*}
\end{proof} From this stability lemma, we obtain a stability result for the trajectories associated with $\mathbf{1}_{\Omg_t}$: \begin{lemma}\label{cor:stab}
There exists an absolute constant $c_1>0$ such that \begin{equation}\label{eq:incl}
\begin{split}
[0,\pi]^2\backslash B(x_c,\frac{3\dlt}{2}) \subset \Phi(t, [0,\pi]^2\backslash B(x_c,\dlt) ) \subset [0,\pi]^2\backslash B(x_c,\frac{\dlt}{2})
\end{split}
\end{equation} for $0\le t \le c_1\eps^{-1}$. 
\end{lemma}
\begin{proof}
	We compute  
 \begin{equation*}
	\begin{split}
	\frac{d}{dt} \psi^{BC}(\Phi(t,x)) = u(t,\Phi(t,x))\cdot\nb\psi^{BC}(\Phi(t,x)) = \left(u(t,\Phi(t,x))-u^{BC}(\Phi(t,x))\right)\cdot\nb\psi^{BC}(\Phi(t,x)),
	\end{split}
	\end{equation*} 
	using stationarity of $\omg^{BC}$. Since   $|\nb \psi^{BC}|$  is uniformly bounded, using Lemma \ref{lem:stab}, we obtain that \begin{equation}\label{eq:con}
	\begin{split}
	\left|\psi^{BC}(\Phi(t,x))-\psi^{BC}(x)\right| \le \int_0^t \left| \frac{d}{ds} \psi^{BC}(\Phi(s,x)) \right| ds \le Ct\eps. 
 	\end{split}
	\end{equation}   
	Equipped with \eqref{eq:con}, let us proceed to prove the second inclusion of \eqref{eq:incl}. We consider two cases.
	
	\medskip
	
	\noindent Case (i):  $x\in  \overline{B(x_c,2\dlt)}\backslash B(x_c,\dlt)$.
	
	\medskip 
	
	\noindent In this case, we have that 
$$ \psi^{BC}(x)-\psi^{BC}(x_c)   \ge  
\frac{|x-x_c|^2}{4} - \frac{|x-x_c|^3}{16} \ge
\frac{\dlt^2}{4}-  \frac{\dlt^3}{2} $$ 
from \eqref{eq:psi-bounds}. 
On the other hand, if we assume in addition that $\Phi(t,x)\in B(x_c,\frac{\dlt}{2})$, this would require that
  $$ \psi^{BC}(x_c)- \psi^{BC}(\Phi(t,x))  \ge -\frac{|\Phi(t,x)-x_c|^2}{4} - \frac{|\Phi(t,x)-x_c|^3}{16}\ge -
 \frac{\dlt^2}{2\color{black}^2 \cdot 4}
  - \frac{\dlt^3}{2\color{black}^3\cdot 16},$$
 which gives (recall that $\delta \le \frac{1}{4}$)\begin{equation*}
	\begin{split}
	\psi^{BC}(x) -\psi^{BC}(\Phi(t,x))\ge
\frac{\dlt^2}{4}-  \frac{\dlt^3}{2} 	 -
 \frac{\dlt^2}{ 2\color{black}^2 \cdot 4}
  - \frac{\dlt^3}{ 2\color{black}^3\cdot 16}=
  \frac{\dlt^2}{4}(1-2\dlt -\frac{1}{ 2\color{black}^2}	 -\frac{\dlt}{ 2^5\color{black}})\ge \frac{\dlt^2}{ 32\color{black}}>0.
	\end{split}
	\end{equation*} 	
	 Taking $c_1>0$ small enough (depending only on $\dlt$ and $C$ from \eqref{eq:con}), this gives a contradiction to \eqref{eq:con} for $t\in [0,c_1\eps^{-1}]$, which shows that 
	$ \Phi(t, \overline{B(x_c,2\dlt)} \backslash B(x_c,\dlt) ) \subset [0,\pi]^2\backslash B(x_c,\frac{\dlt}{ 2\color{black}})$.

	\medskip
	
	\noindent Case (ii): $x\in [0,\pi]^2\backslash \overline{B(x_c,2\dlt)}$.
	
	\medskip 
	
	\noindent If $\Phi(t,x)\notin  {B(x_c, \dlt/2 )}$ for all $t\in [0,c_1\eps^{-1}]$, then there is nothing to prove. So we may assume  $|\Phi(t_0,x)-x_c|<\dlt/2 $  for some $t_0\in (0,c_1\eps^{-1}]$.
	 But then, by continuity of the flow, we must have $|\Phi(t_1,x)-x_c|=2\dlt $  for some $t_1\in(0,t_0)$. Thus, by replacing $x$ with 
$\Phi(t_1,x)$	  in \eqref{eq:con}, we obtain that 
$$\frac{\dlt^2}{ 32 }\leq C(t_1-t_0)\epsilon\leq 
Cc_1.$$ 
This is  a contradiction due to the choice of $c_1>0$. 
	 
	 \medskip

	 Therefore, the second inclusion in \eqref{eq:incl} is proved. The proof for the other inclusion can be done similarly.
\end{proof}
We now state a complementary lemma: \begin{lemma}\label{lem:stab2}
	We have that \begin{equation*}
	\begin{split}
	\Phi(t,B(x_o,\widetilde{\eps}))\subset B(x_o,\dlt), \quad t \in [0,c_1\eps^{-1}],
	\end{split}
	\end{equation*} where $\widetilde{\eps}:= (c_2\dlt)^{\exp((c_3\eps)^{-1})}$ with  
	some constants $c_2\in(0,1]$, $c_3>0$ 	
	depending only on $c_1$. 
\end{lemma}
\begin{proof}
	The conclusion follows directly from the trajectory estimate for 2D Euler (see \cite[Lemma 2]{EJ}) \begin{equation*}
	\begin{split}
	|\Phi(t,x)-\Phi(t,x')|\le C|x-x'|^{\exp(-Ct\nrm{\omg_0}_{L^\infty})}
	\end{split}
	\end{equation*} upon taking $x' = x_o$. 
\end{proof}

\begin{proof}[Proof of Theorem \ref{thm_finite}]
	From  the assumptions on $\Omg_0$, we   take an injective 
 piecewise differentiable  
	curve $\gmm_0:[0,1]\rightarrow(0,\pi)^2$ satisfying \begin{itemize}
		\item $|\gmm_0(0)-x_o| = \widetilde{\eps}$ and $|\gmm_0(1) - x_c| = \dlt,$
		\item $\gmm_0([0,1]) \subset (0,\pi)^2\backslash B(x_c,\dlt)$ and $\gmm_0([0,1]) \subset \partial\Omg_0$,
		\item $|N_{x_c}[\gamma_0]|\leq 1.$
	\end{itemize} We set $x_o^* =\gmm_0(0)$, $x_c^* = \gmm_0(1)$, and $I=[0,c_1\eps^{-1}]$. 
 In the following,   $t$ will be assumed to lie on the interval $I$.	
	From Lemma \ref{cor:stab} and Lemma \ref{lem:stab2}, it follows that 
	\begin{itemize}
		\item $|\Phi(t,x_o^*)-x_o|\le\dlt$ and $\frac{\dlt}{2} \le |\Phi(t,x_c^*)-x_c| \le \frac{3\dlt}{2}$, and 
		\item $\Phi(t,\gmm_0)=:\gmm_t$ does not intersect $B(x_c,\frac{\dlt}{2})$. 
	\end{itemize} 

By \eqref{iden_wind}, we observe  that the winding number $N_{x_c}[\gamma_{t}]$ of the curve $\gmm_{t}$ around $x_c$ satisfies \begin{equation*}
	\begin{split}
		N_{x_c}[\gamma_{t}]+1\ge  N_{x_c}[\gamma_{t}] - N_{x_c}[\gamma_0] \ge  \frac{1}{2\pi\cdot(3/2)\dlt}\int_0^{t} u(t',\Phi(t',x_c^*)) \cdot 
\frac{\left(\Phi(t',x_c^*)-x_c\right)^\perp}{|\Phi(t',x_c^*)-x_c|}	
		 dt'  -c.
	\end{split}
	\end{equation*} 
	The  constant $c$ in the above appears from a potential loss of the winding number due to motion of the other endpoint $\Phi(t,x_o^*)$, which is bounded above by 1 since $\Phi(t,x_o^*)$ cannot escape the ball $B(x_o,\dlt)$.

On the other hand, by using 
\eqref{eq:u-bounds} and Lemma \ref{lem:stab},
we have
 \begin{equation*}
	\begin{split}
	&\int_0^{t} u(t',\Phi(t',x_c^*)) \cdot \frac{\left(\Phi(t',x_c^*)-x_c\right)^\perp}{|\Phi(t',x_c^*)-x_c|}	  dt' 
	 \ge \int_0^{t} \frac 1 2   {\left(\Phi(t',x_c^*)-x_c\right)^\perp} \cdot \frac{\left(\Phi(t',x_c^*)-x_c\right)^\perp}{|\Phi(t',x_c^*)-x_c|}	  dt' \\ &\quad 
	 -
	 \int_0^{t} \Big|\frac 1 2  {\left(\Phi(t',x_c^*)-x_c\right)^\perp} -u^{BC}(\Phi(t',x_c^*))\Big|dt'
	  -
	 \int_0^{t} \Big| u^{BC}(\Phi(t',x_c^*))-u(t',\Phi(t',x_c^*))\Big|dt'\\
	 &\quad \ge \frac 1 2 \int_0^t {|\Phi(t',x_c^*)-x_c|}	 dt' - \frac 1 4 
	  \int_0^t {|\Phi(t',x_c^*)-x_c|^2}	 dt'
	   - C\eps
	  \int_0^t1 dt'\\
	 &\quad 
	     \ge c_*t\cdot(\delta-C\eps).
	\end{split}
	\end{equation*} 
	 We shall proceed assuming $\eps>0$ small enough so that $C\eps <\frac{\dlt}{2}$. 
		\color{black}	 Thus we get
\begin{equation*}
	\begin{split}
		N_{x_c}[\gamma_{t}]+2\ge    c_*' t,\quad t\in I.
	\end{split}
	\end{equation*} 		
	By \eqref{result_wind}, we conclude that \begin{equation*}
	\begin{split}
	\mathrm{length}(\partial\Omg_{t}) \ge \mathrm{length}(\gmm_{t}) \ge \max\left\{\tilde{c}, (c_*' t-2)\cdot\left(2\pi \cdot \frac{\dlt}{2}\right) \right\},\quad t\in I,
	\end{split} 
	\end{equation*} where the constant
$$\tilde{c}:=\dist\left(B(x_c, (3/2)\dlt), B(x_0,\dlt)\right)>0$$ is independent of $\eps>0$. Hence we get $\mathrm{length}( \partial \Omg_t ) \ge c_0 t
$ for $t \in I$ where $c_0$ is a universal constant.
\end{proof}

\section{The whole space case}\label{sec:R2}

\begin{proof}[Proof of Theorem \ref{thm_finite_whole}]
We denote $D=B(x_o,1)\subset\mathbb{R}^2$. Recall that $u_{D} = K * \mathbf{1}_D$  satisfies  
\begin{equation}\label{vel_disk}
u_{D,rad}\equiv 0,\quad u_{D,tan}(x)=
\begin{cases} \frac{|x|}{2} &\quad\mbox{if } x\in D,\\
\frac{1}{2|x|} &\quad\mbox{otherwise }  
\end{cases} \end{equation}
for the decomposition $u =u_{rad}\frac{x}{|x|}+u_{tan}\frac{x^\perp}{|x|}.$ From the assumption of the theorem, we   take  
 an injective piecewise differentiable  
  curve $\gmm_0:[0,1]\to (\overline{B(x_o,2)}\setminus D) \cap \partial\Omega_0$  satisfying $\gmm_0(0)=x_1$ with $|x_1|=1,$ and  $\gmm_0(1)=x_2$ with $|x_2|=2$ and
$$|N_{x_0}[\gmm_0]|\leq 1.$$  
Let's denote
$|\Omega_0 \sd D |=\dlt$. 
   By the stability result of \cite{SiVe} (also see \cite{wp} for bounded domains), we have
$$|\Omega_t \sd D |\le C \sqrt{\delta},\quad t\geq 0,$$
where 
$\mathbf{1}_{\Omega_t}$ is the 
solution of \eqref{eq:Euler} for the initial data $\mathbf{1}_{\Omega_0}$.
As 
 in  the proof of Lemma \ref{lem:stab}, 
we get, for $u(t) = K* \mathbf{1}_{\Omega_t}$, \begin{equation*}
	\begin{split}
	\nrm{u(t) - u_{D}}_{L^\infty} \le C_1 \delta^{1/4},\quad t\geq 0,
	\end{split}
	\end{equation*}  where $C_1>0$ is a universal constant  (e.g. see \cite[Lemma 2.2]{CJ_winding}). We then observe that
	$$  |u_{rad}(t,x)|\le C_1\delta^{1/4},\quad t\geq0,\quad x\in\mathbb{R}^2$$ follows from \eqref{vel_disk}. 	 
Since $|x_1(0)|=1, |x_2(0)|=2$, we have, for $$t\in [0, \frac 1 4 / {( C_1 \delta^{1/4}})]=:I,$$
we have	$x_1(t)\in \overline{B_{5/4}}\setminus B_{3/4},$ and
$x_2(t)\in \overline{B_{9/4}}\setminus B_{7/4}$, where $x_i(t) := \Phi(t,x_i)$ for $i = 1,2$. In the following,   $t$ will be assumed to lie on the interval $I$. We observe that for any $x \in \bbR^2\backslash D,$  
	\begin{equation*}\Phi(t,x) \in \bbR^2\backslash B(x_o,{\frac 3 4}).\end{equation*}
Thus, we obtain from \eqref{vel_disk} that
$$ u_{tan}(t, x_1(t))\ge \min\{\frac{3/4}{2},\frac{1}{2\cdot 5/4}\}- C_1\delta^{1/4}=\frac {21}{56} - C_1\delta^{1/4},\qquad u_{tan}(t, x_2(t))\le \frac{1}{2\cdot 7/4} + C_1\delta^{1/4}=\frac{16}{ 56} + C_1\delta^{1/4}.$$  
Assume $\delta>0$ small enough so that $C_1\delta^{1/4}\leq \frac 1 {56}.$ Then 
$$ u_{tan}(t, x_1(t))\ge   \frac {20}{56},\quad u_{tan}(t, x_2(t))\le \frac{17}{ 56}.$$  
Thus we have the following   estimate  on the winding number $N_{x_0}[\gmm_t]$ for the curve $\gmm_t:=\Phi(t,\gmm_0)$ around $x_o$ by \eqref{iden_wind}:
\begin{equation*}
\begin{split}
N_{x_0}[\gmm_t]  + 1&\ge 
N_{x_0}[\gmm_t]  -N_{x_0}[\gmm_0]
\ge
 \frac{1}{2\pi\cdot(5/4)}\int_0^t  u_{tan}(t', x_1(t'))     dt' 
-\frac{1}{2\pi\cdot(7/4)}\int_0^t u_{tan}(t', x_2(t'))   dt'  \\
 &\ge \frac{20}{56}\cdot \frac{t}{2\pi \cdot (5/4)} - \frac{17}{56}\cdot \frac{t}{2\pi \cdot (7/4) }=  c_1 t,
\end{split}
\end{equation*}  where $c_1>0$ is an absolute constant.  This gives 
 \begin{equation*}
 \mathrm{length}( \partial \Omg_t ) \ge \mathrm{length}(\Phi(t,l)) \ge  \max\left\{\left(\frac 7 4 - \frac 5 4\right), (c_1t-1) \cdot \left(2\pi\cdot \frac 3 4\right) \right\}.
 \end{equation*} Hence, there is an absolute constant $c_0>0$ such that $\mathrm{length}( \partial \Omg_t ) \ge c_0 t
$ for $t \in I$ and the proof is done by setting 
$c_1=(4C_1)^{-4}$. 
\end{proof}

{\Large \section*{Acknowledgement}}

\noindent KC has been supported by   the National Research Foundation of Korea (NRF-2018R1D1A1B07043065) and by the Research Fund (1.200085.01) of UNIST(Ulsan National Institute of Science \& Technology).
IJ has been supported  by a KIAS Individual Grant MG066202 at Korea Institute for Advanced Study, the Science Fellowship of POSCO TJ Park Foundation, and the National Research Foundation of Korea grant 2019R1F1A1058486. 

\ \\

\bibliography{patch11}

\end{document}